\documentclass[11pt]{article}

\usepackage[a4paper,margin=1in]{geometry}
\usepackage{amsmath,amssymb,amsthm,mathtools}
\usepackage{enumitem}
\usepackage{hyperref}
\usepackage{microtype}
\usepackage{mathrsfs}

\hypersetup{
  colorlinks=true,
  linkcolor=blue,
  citecolor=blue,
  urlcolor=blue
}

\newtheorem{theorem}{Theorem}[section]
\newtheorem{lemma}[theorem]{Lemma}
\newtheorem{corollary}[theorem]{Corollary}
\theoremstyle{definition}
\newtheorem{definition}[theorem]{Definition}
\newtheorem{problem}[theorem]{Problem}
\theoremstyle{remark}
\newtheorem{remark}[theorem]{Remark}

\newcommand{\D}{\mathbb D}
\newcommand{\T}{\mathbb T}
\newcommand{\C}{\mathbb C}

\newcommand{\ds}{\,ds}

\newcommand{\QR}{\operatorname{QR}}
\newcommand{\Car}{\operatorname{Car}}

\title{Gabriel-Type Estimates for Harmonic Quasiregular Mappings and Stoilow Classes}
\author{
  Elver Bajrami\\
  \small Department of Mathematics, University of Prishtina, Prishtina, Kosovo\\
  \small \texttt{elver.bajrami@uni-pr.edu}
}
\date{}

\begin{document}
\maketitle

\begin{abstract}
We establish a dilatation-dependent Gabriel inequality for sense-preserving harmonic $K$-quasiregular mappings in $h^2$. If $f=h+\overline g\in h^2$ satisfies $|g'|\le k|h'|$, where $K=(1+k)/(1-k)$, then every convex curve $\Gamma\subset\D$ satisfies
\[
   \int_\Gamma |f(z)|^2\ds(z)
   \le 2\frac{(1+k)^2}{1+k^2}
      \int_\T |f^*(\zeta)|^2\,|d\zeta|
   =\frac{4K^2}{K^2+1}
      \int_\T |f^*(\zeta)|^2\,|d\zeta|.
\]
The coefficient recovers the sharp analytic constant $2$ at $K=1$, is strictly smaller than the general harmonic constant $4$ for every finite $K$, and tends to $4$ as $K\to\infty$. The proof retains quantitative information on the analytic and co-analytic parts through the Hardy--Stein identity and the complex dilatation bound. To place this estimate in a broader quasiregular setting, we develop a maximal-function principle based on the Carleson-measure geometry of convex curves. We also identify a log-subharmonic modulus class that inherits the sharp analytic constant and prove Stoilow-factorization criteria under explicit boundary-weight and cone-distortion assumptions. These results distinguish distortion-dependent estimates from those arising through general boundary maximal control. The resulting formulation clarifies which conclusions follow from harmonic quasiregularity and which require separate boundary regularity assumptions.
\end{abstract}


\section{Introduction}

Let
\[
   \D=\{z\in\C: |z|<1\}, \qquad \T=\partial\D,
\]
and let $|d\zeta|$ denote arclength measure on $\T$. For $0<p<\infty$, the analytic Hardy space $H^p$ consists of all analytic functions $f$ in $\D$ for which
\[
   \|f\|_{H^p}^p
   :=\sup_{0<r<1}\int_\T |f(r\zeta)|^p\,|d\zeta|<\infty.
\]
The harmonic Hardy space $h^p$ is defined analogously for complex-valued harmonic mappings. Gabriel proved that if $f\in H^p$ and $\Gamma$ is a convex curve in $\D$, then
\begin{equation}\label{eq:Gabriel}
   \int_\Gamma |f(z)|^p\ds(z)
   \le 2\int_\T |f^*(\zeta)|^p\,|d\zeta|.
\end{equation}
Here $f^*$ is the radial boundary function. The constant $2$ is sharp; see \cite{Gabriel1928,Duren1970}.

A complex-valued harmonic analogue was recently obtained by Das \cite{Das2025}. If $f\in h^p$, $p>1$, and $\Gamma\subset\D$ is convex, then
\begin{equation}\label{eq:Das}
   \int_\Gamma |f(z)|^p\ds(z)
   \le A_h(p)\int_\T |f^*(\zeta)|^p\,|d\zeta|,
\end{equation}
where
\begin{equation}\label{eq:Ah}
   A_h(p)=
   \begin{cases}
      4, & p\ge2,\\[3pt]
      2\sec^p\!\left(\dfrac{\pi}{2p}\right), & 1<p<2.
   \end{cases}
\end{equation}
Das also showed that the convex-curve inequality does not extend in this general form to the full range $0<p\le1$.

The constant $4$ in the case $p=2$ applies to the whole harmonic Hardy space and therefore does not reflect any quantitative control of the analytic and co-analytic parts. The principal aim of this paper is to show that bounded complex dilatation yields an explicit improvement. Recall that a complex-valued harmonic mapping has a decomposition
\[
   f=h+\overline g,
\]
where $h$ and $g$ are analytic. If $f$ is sense-preserving and
\[
   |g'(z)|\le k|h'(z)|, \qquad 0\le k<1,
\]
then $f$ is $K$-quasiregular with $K=(1+k)/(1-k)$.

\begin{theorem}[Dilatation-dependent Gabriel inequality]\label{thm:main}
Let $f=h+\overline g\in h^2$ be a sense-preserving harmonic mapping in $\D$ satisfying
\[
   |g'(z)|\le k|h'(z)|\qquad (0\le k<1).
\]
Choose the decomposition normalized by $g(0)=0$, and put $K=(1+k)/(1-k)$. Then every convex curve $\Gamma\subset\D$ satisfies
\begin{equation}\label{eq:main-k}
   \int_\Gamma |f(z)|^2\ds(z)
   \le 2\frac{(1+k)^2}{1+k^2}
      \int_\T |f^*(\zeta)|^2\,|d\zeta|,
\end{equation}
and equivalently
\begin{equation}\label{eq:main-K}
   \int_\Gamma |f(z)|^2\ds(z)
   \le \frac{4K^2}{K^2+1}
      \int_\T |f^*(\zeta)|^2\,|d\zeta|.
\end{equation}
\end{theorem}

The coefficient $4K^2/(K^2+1)$ is strictly increasing for $K\ge1$, equals $2$ at $K=1$, and tends to $4$ as $K\to\infty$. Hence Theorem \ref{thm:main} gives a strict improvement over the general harmonic constant for every finite $K$. The proof, given in Section~\ref{sec:harmonic-main}, is not a formal restriction of \eqref{eq:Das}: it combines the exact $h^2$ boundary-energy splitting of the analytic and co-analytic parts with the estimate imposed by the complex dilatation. We do not claim sharpness for $1<K<\infty$; determining the optimal dependence on $K$ remains open.

We next place the main estimate in a broader quasiregular framework. A mapping $f:\D\to\C$ is $K$-quasiregular if $f\in W^{1,2}_{\mathrm{loc}}(\D)$, $J_f\ge0$ a.e., and
\begin{equation}\label{eq:qr}
   |Df(z)|^2\le KJ_f(z) \qquad \text{for a.e. }z\in\D.
\end{equation}
Equivalently, in complex notation, $|f_{\bar z}|\le k|f_z|$ a.e., where $K=(1+k)/(1-k)$. Background on quasiregular and quasiconformal mappings may be found in \cite{Vaisala1971,AstalaKoskela2011,AdamowiczGonzalez2021,AdamowiczGonzalez2025}.

The most direct radial analogue of $H^p$ is the class
\begin{equation}\label{eq:radial-class}
   \mathcal H^p_{qr,K}(\D)
   =\left\{f\in\QR_K(\D):
     \sup_{0<r<1}\int_\T |f(r\zeta)|^p\,|d\zeta|<\infty\right\}.
\end{equation}
Unlike analytic Hardy-space functions, general quasiregular mappings satisfying \eqref{eq:radial-class} need not possess almost everywhere boundary values. Consequently, a boundary-norm inequality with right-hand side
\[
   \int_\T |f^*(\zeta)|^p\,|d\zeta|
\]
requires additional boundary information. This obstruction motivates the maximal-function and Stoilow formulations developed below.

The paper complements Theorem \ref{thm:main} in three directions.
\begin{enumerate}[label=(\roman*)]
\item We establish a maximal-function principle: non-tangential boundary values together with an $L^p$ non-tangential maximal estimate imply Gabriel-type inequalities through the Carleson-measure geometry of convex curves.
\item For mappings with log-subharmonic modulus, the Lozi\'nski majorant theorem transfers Gabriel's sharp analytic constant $2$ to the corresponding quasiregular subclass.
\item For mappings represented by Stoilow factorization, we prove boundary-norm estimates under explicit boundary-weight and cone-distortion hypotheses, thereby identifying sufficient mechanisms for transporting analytic estimates through a quasiconformal change of variables.
\end{enumerate}
These statements separate the quantitative role of the dilatation in Theorem \ref{thm:main} from the general boundary-control mechanisms that operate in wider quasiregular classes.

Gabriel's problem is also related to Riesz-type inequalities for quasiregular functions. When the convex curve is specialized to a concentric circle, the resulting estimates are connected with classical circle-mean inequalities in quasiregular Hardy spaces; see \cite{Kalaj2025}.

Sections~2 and~3 develop the Carleson and maximal-function framework. Section~4 treats the log-subharmonic modulus case. Section~\ref{sec:harmonic-main} proves Theorem \ref{thm:main}. Section~6 gives the Stoilow criteria, and Section~7 lists open problems.

\section{Carleson geometry of convex curves}

For an arc $I\subset\T$, let $|I|$ be its arclength and define the Carleson box
\[
   S(I)=\left\{re^{it}: e^{it}\in I,
   \ 1-\frac{|I|}{2\pi}\le r<1\right\}.
\]
A finite positive Borel measure $\mu$ on $\D$ is a Carleson measure if
\[
   \|\mu\|_{\Car}:=\sup_{I\subset\T}\frac{\mu(S(I))}{|I|}<\infty.
\]
Throughout the paper, a convex curve means a simple rectifiable arc, or a closed rectifiable curve, contained in the boundary of a convex subset of $\D$; line segments are allowed. This convention excludes repeated parametrizations, since the relevant measure is arclength on the trace.

\begin{lemma}[Arclength on convex curves is Carleson]\label{lem:convex-carleson}
There is an absolute constant $C_0$ such that, for every convex curve $\Gamma\subset\D$, the measure
\[
   d\mu_\Gamma=\ds|_\Gamma
\]
is Carleson and
\[
   \|\mu_\Gamma\|_{\Car}\le C_0.
\]
\end{lemma}

\begin{proof}
Let $I\subset\T$. If $|I|\ge1$, then
\[
   \mu_\Gamma(S(I))\le \ell(\Gamma)\le 2\pi,
\]
since the perimeter of a convex set contained in the unit disk is at most the perimeter of the unit disk. Hence $\mu_\Gamma(S(I))\le 2\pi |I|$ in this case.

Suppose now that $0<|I|<1$. The box $S(I)$ is contained in a Euclidean disk $B(\xi_I,c|I|)$, where $\xi_I$ is the midpoint of $I$ and $c$ is absolute. We use the elementary fact that the length of a convex curve contained in a disk of radius $R$ is bounded by $2\pi R$. More generally, the part of a convex curve lying in such a disk has length at most $C R$ with an absolute constant $C$; this follows by applying Cauchy's projection formula to the convex hull of that part. Consequently
\[
   \mu_\Gamma(S(I))\le Cc|I|.
\]
The two estimates give the desired uniform Carleson bound.
\end{proof}

\begin{remark}
No attempt is made to optimize $C_0$. The sharp constant in Gabriel's analytic theorem does not follow from this Carleson-measure argument; the lemma is used only to obtain robust finite bounds in non-sharp quasiregular settings.
\end{remark}

Fix $\alpha>1$ and define the Stolz cone
\[
   \Gamma_\alpha(\xi)=\{z\in\D: |z-\xi|<\alpha(1-|z|)\},
   \qquad \xi\in\T.
\]
For a function $u$ on $\D$, set
\[
   N_\alpha u(\xi)=\sup_{z\in\Gamma_\alpha(\xi)} |u(z)|.
\]

\begin{lemma}[Carleson embedding through non-tangential maximal functions]\label{lem:max-embed}
Let $0<p<\infty$ and $\alpha>1$. There exists $C_\alpha<\infty$ such that, for every Carleson measure $\mu$ on $\D$ and every measurable function $u$ on $\D$,
\begin{equation}\label{eq:max-embed}
   \int_\D |u(z)|^p\,d\mu(z)
   \le C_\alpha\|\mu\|_{\Car}
      \int_\T N_\alpha u(\xi)^p\,|d\xi|.
\end{equation}
\end{lemma}

\begin{proof}
Let $E_t=\{\xi\in\T:N_\alpha u(\xi)>t\}$. If $|u(z)|>t$, then $z$ belongs to the tent over $E_t$ associated with the aperture $\alpha$. The Carleson condition gives
\[
   \mu(\{z:|u(z)|>t\})
   \le C_\alpha\|\mu\|_{\Car}|E_t|.
\]
Using the layer-cake formula,
\begin{align*}
   \int_\D |u|^p\,d\mu
   &=p\int_0^\infty t^{p-1}\mu(\{|u|>t\})\,dt \\
   &\le C_\alpha\|\mu\|_{\Car}
      p\int_0^\infty t^{p-1}|E_t|\,dt \\
   &=C_\alpha\|\mu\|_{\Car}
      \int_\T N_\alpha u(\xi)^p\,|d\xi|.
\end{align*}
\end{proof}

Lemma \ref{lem:convex-carleson} and Lemma \ref{lem:max-embed} immediately imply the following curve estimate.

\begin{corollary}\label{cor:curve-max}
For every $0<p<\infty$ and every $\alpha>1$ there is a constant $B_\alpha<\infty$ such that, for every convex curve $\Gamma\subset\D$ and every measurable function $u$ on $\D$,
\begin{equation}\label{eq:curve-max}
   \int_\Gamma |u(z)|^p\ds(z)
   \le B_\alpha\int_\T N_\alpha u(\xi)^p\,|d\xi|.
\end{equation}
\end{corollary}

\section{A maximal-function formulation for quasiregular classes}

The preceding section gives a simple general principle: to prove a Gabriel inequality, it is enough to control a non-tangential maximal function by the boundary function.

\begin{definition}\label{def:admissible}
Let $0<p<\infty$, $K\ge1$, $\alpha>1$, and $M<\infty$. We say that a $K$-quasiregular mapping $f:\D\to\C$ belongs to
\[
   \mathcal G^p_{K,\alpha}(M)
\]
if the following hold:
\begin{enumerate}[label=(\roman*)]
\item $f$ has a non-tangential boundary limit
\[
   f^*(\xi)=\lim_{z\to\xi,\,z\in\Gamma_\alpha(\xi)}f(z)
\]
for a.e. $\xi\in\T$;
\item $f^*\in L^p(\T)$;
\item the maximal estimate
\begin{equation}\label{eq:admissible-max}
   \|N_\alpha f\|_{L^p(\T)}
   \le M\|f^*\|_{L^p(\T)}
\end{equation}
holds.
\end{enumerate}
\end{definition}

\begin{theorem}[Maximal-function Gabriel principle]\label{thm:admissible}
Let $0<p<\infty$, $K\ge1$, $\alpha>1$, and $M<\infty$. If $f\in\mathcal G^p_{K,\alpha}(M)$, then every convex curve $\Gamma\subset\D$ satisfies
\begin{equation}\label{eq:admissible-gabriel}
   \int_\Gamma |f(z)|^p\ds(z)
   \le B_\alpha M^p
      \int_\T |f^*(\xi)|^p\,|d\xi|,
\end{equation}
where $B_\alpha$ is the constant in Corollary \ref{cor:curve-max}. More generally, if $\mu$ is a Carleson measure on $\D$, then
\begin{equation}\label{eq:admissible-carleson}
   \int_\D |f(z)|^p\,d\mu(z)
   \le C_\alpha\|\mu\|_{\Car}M^p
      \int_\T |f^*(\xi)|^p\,|d\xi|.
\end{equation}
\end{theorem}

\begin{proof}
Apply Lemma \ref{lem:max-embed} to $u=f$ and then use \eqref{eq:admissible-max}. This gives \eqref{eq:admissible-carleson}. Taking $\mu=\ds|_\Gamma$ and using Lemma \ref{lem:convex-carleson} gives \eqref{eq:admissible-gabriel}.
\end{proof}

\begin{remark}\label{rem:scope}
Theorem \ref{thm:admissible} isolates a precise reduction from a boundary maximal estimate to a Gabriel-type inequality. The required maximal estimate may follow from additional structure of the mapping class or from an independent boundary theorem. In particular, the radial condition \eqref{eq:radial-class} alone does not generally provide almost everywhere boundary values or non-tangential maximal control.
\end{remark}

\section{The log-subharmonic modulus case}

A non-negative function $\Phi$ is called log-subharmonic if $\log\Phi$ is subharmonic, with the convention $\log0=-\infty$.

\begin{theorem}\label{thm:log-subharmonic}
Let $0<p<\infty$, and let $f$ be a mapping in $\D$ such that $\Phi=|f|$ is log-subharmonic and
\begin{equation}\label{eq:Phi-Hp}
   \sup_{0<r<1}\int_\T \Phi(r\zeta)^p\,|d\zeta|<\infty.
\end{equation}
Then every convex curve $\Gamma\subset\D$ satisfies
\begin{equation}\label{eq:log-sub-gabriel}
   \int_\Gamma |f(z)|^p\ds(z)
   \le 2\int_\T \Phi^*(\zeta)^p\,|d\zeta|,
\end{equation}
where $\Phi^*$ is the radial boundary function supplied by the Lozi\'nski majorant theorem. If $f$ itself has radial boundary values a.e., then
\begin{equation}\label{eq:log-sub-gabriel-boundary}
   \int_\Gamma |f(z)|^p\ds(z)
   \le 2\int_\T |f^*(\zeta)|^p\,|d\zeta|.
\end{equation}
The constant $2$ is best possible for any subclass containing the analytic Hardy space $H^p$.
\end{theorem}

\begin{proof}
By the Lozi\'nski majorant theorem \cite{Lozinski1944}, there exists an analytic function $F\in H^p$ such that
\[
   \Phi(z)\le |F(z)| \qquad (z\in\D)
\]
and
\[
   \Phi^*(\zeta)=|F^*(\zeta)| \qquad \text{for a.e. }\zeta\in\T.
\]
Gabriel's theorem gives
\begin{align*}
   \int_\Gamma |f(z)|^p\ds(z)
   &=\int_\Gamma \Phi(z)^p\ds(z) \\
   &\le \int_\Gamma |F(z)|^p\ds(z) \\
   &\le 2\int_\T |F^*(\zeta)|^p\,|d\zeta| \\
   &=2\int_\T \Phi^*(\zeta)^p\,|d\zeta|.
\end{align*}
If $f^*$ exists a.e., then $\Phi^*=|f^*|$ a.e. Sharpness follows from the sharpness of Gabriel's analytic theorem.
\end{proof}

\begin{corollary}\label{cor:qr-log-sub}
Let $0<p<\infty$ and $K\ge1$. Suppose that $f$ is $K$-quasiregular in $\D$, satisfies \eqref{eq:Phi-Hp} with $\Phi=|f|$, has radial boundary values a.e., and $\log|f|$ is subharmonic. Then every convex curve $\Gamma\subset\D$ satisfies
\[
   \int_\Gamma |f(z)|^p\ds(z)
   \le 2\int_\T |f^*(\zeta)|^p\,|d\zeta|.
\]
The constant $2$ is sharp in any such class that contains all analytic $H^p$ functions.
\end{corollary}

\begin{remark}
No quasiregularity is used in Theorem \ref{thm:log-subharmonic}. The point of Corollary \ref{cor:qr-log-sub} is that the quasiregular Gabriel problem has a sharp answer on the special subclass for which the modulus admits an analytic majorant with the same boundary modulus. Related uses of Lozi\'nski-type majorants in isoperimetric inequalities for subharmonic or harmonic functions appear in \cite{Mateljevic2015,MateljevicPurtic2025}.
\end{remark}

\section{Harmonic quasiregular mappings and proof of the main theorem}\label{sec:harmonic-main}

Let $f=h+\overline g$ be a complex-valued harmonic mapping in $\D$, where $h$ and $g$ are analytic. If $f$ is sense-preserving and
\[
   |g'(z)|\le k|h'(z)|, \qquad 0\le k<1,
\]
then $f$ is $K$-quasiregular with
\[
   K=\frac{1+k}{1-k}.
\]
For comparison, Das's theorem immediately gives, for every harmonic $K$-quasiregular mapping $f\in h^p$ with $p>1$,
\begin{equation}\label{eq:Das-harmonic-qr}
   \int_\Gamma |f(z)|^p\ds(z)
   \le A_h(p)\int_\T |f^*(\zeta)|^p\,|d\zeta|
\end{equation}
for every convex curve $\Gamma\subset\D$, where $A_h(p)$ is defined in \eqref{eq:Ah}. The estimate \eqref{eq:Das-harmonic-qr} is the restriction of the general harmonic result to this subclass and serves as the baseline for the distortion-dependent improvement at $p=2$.

\begin{proof}[Proof of Theorem \ref{thm:main}]
Replace $g$ by $g-g(0)$ and $h$ by $h+\overline{g(0)}$. This leaves $f=h+\overline g$ and the derivatives unchanged and gives the normalization $g(0)=0$. Since $f\in h^2$, the normalized analytic parts $h$ and $g$ belong to $H^2$.

Set
\[
   H=\int_\T |h^*(\zeta)|^2\,|d\zeta|,
   \qquad
   G=\int_\T |g^*(\zeta)|^2\,|d\zeta|.
\]
The Hardy--Stein identity gives
\[
   G=4\int_\D |g'(z)|^2\log\frac1{|z|}\,dA(z),
\]
where $dA$ denotes planar area measure, and
\[
   H-2\pi |h(0)|^2
   =4\int_\D |h'(z)|^2\log\frac1{|z|}\,dA(z).
\]
Using $|g'|\le k|h'|$, we obtain
\begin{equation}\label{eq:G-kH}
   G\le k^2\bigl(H-2\pi |h(0)|^2\bigr)\le k^2H.
\end{equation}
Moreover,
\begin{equation}\label{eq:boundary-splitting}
   \int_\T |f^*(\zeta)|^2\,|d\zeta|=H+G.
\end{equation}
Indeed, the cross term vanishes because
\[
   2\operatorname{Re}\int_\T h^*(\zeta)g^*(\zeta)\,|d\zeta|
   =4\pi\operatorname{Re}\bigl(h(0)g(0)\bigr)=0.
\]

Gabriel's theorem applied separately to $h$ and $g$ yields
\[
   \int_\Gamma |h|^2\ds\le 2H,
   \qquad
   \int_\Gamma |g|^2\ds\le 2G.
\]
Consequently, by the triangle inequality and Cauchy--Schwarz,
\begin{align*}
   \int_\Gamma |f|^2\ds
   &\le \int_\Gamma (|h|+|g|)^2\ds \\
   &\le 2H+2G
      +2\left(\int_\Gamma |h|^2\ds\right)^{1/2}
        \left(\int_\Gamma |g|^2\ds\right)^{1/2} \\
   &\le 2(\sqrt H+\sqrt G)^2.
\end{align*}
If $H=0$, then $h\equiv0$, and the derivative inequality forces $g$ to be constant; the normalization gives $g\equiv0$, so the conclusion is immediate. Assume $H>0$ and put $t=\sqrt{G/H}$. By \eqref{eq:G-kH}, $0\le t\le k<1$. Since
\[
   t\longmapsto \frac{(1+t)^2}{1+t^2}
\]
is increasing on $[0,1]$, equations \eqref{eq:boundary-splitting} and the preceding curve estimate give
\[
   \frac{\displaystyle\int_\Gamma |f|^2\ds}
        {\displaystyle\int_\T |f^*|^2\,|d\zeta|}
   \le 2\frac{(1+t)^2}{1+t^2}
   \le 2\frac{(1+k)^2}{1+k^2}.
\]
This proves \eqref{eq:main-k}. Finally, the identity
\[
   2\frac{(1+k)^2}{1+k^2}
   =\frac{4K^2}{K^2+1},
   \qquad K=\frac{1+k}{1-k},
\]
gives \eqref{eq:main-K}.
\end{proof}

\begin{remark}[Where the dilatation enters]\label{rem:dilatation}
The improvement in Theorem \ref{thm:main} is not obtained by merely restricting Das's inequality to the quasiregular subclass. The dilatation condition first yields the energy comparison $G\le k^2H$, while the normalization $g(0)=0$ gives the exact identity \eqref{eq:boundary-splitting}. Retaining both pieces of information produces the factor $2(1+k)^2/(1+k^2)$ and hence the explicit dependence on $K$.
\end{remark}

\begin{remark}[Comparison with circle estimates]
When $\Gamma$ is a circle contained in $\D$, Das's circle theorem gives sharper specialized estimates. Those statements pass unchanged to the harmonic quasiregular subclass and are therefore not restated as separate results here; see \cite[Theorem~4]{Das2025}.
\end{remark}

\section{Stoilow factorizations with boundary control}

Every non-constant planar quasiregular mapping admits a Stoilow factorization
\begin{equation}\label{eq:Stoilow}
   f=g\circ\phi,
\end{equation}
where $g$ is analytic and $\phi$ is quasiconformal. In the disk setting, after normalization, one may take $\phi:\D\to\D$. To obtain a boundary-norm Gabriel inequality from \eqref{eq:Stoilow}, one needs quantitative control of the boundary map
\[
   \psi=\phi|_\T.
\]
The following result states this control explicitly.

\begin{definition}\label{def:weight}
Let $\psi:\T\to\T$ be an orientation-preserving homeomorphism. Suppose that the push-forward measure $\psi_*(|d\xi|)$ is absolutely continuous with respect to arclength. Its density $w_\psi$ is defined by
\[
   \int_\T H(\eta)w_\psi(\eta)\,|d\eta|
   =\int_\T H(\psi(\xi))\,|d\xi|
\]
for all non-negative Borel functions $H$.
For $1<p<\infty$, we say that $w_\psi\in A_p(\T)$ if
\[
   [w_\psi]_{A_p}
   :=\sup_{I\subset\T}
   \left(\frac1{|I|}\int_I w_\psi\,|d\xi|\right)
   \left(\frac1{|I|}\int_I w_\psi^{-1/(p-1)}\,|d\xi|\right)^{p-1}<\infty.
\]
\end{definition}

\begin{theorem}[Weighted Stoilow criterion]\label{thm:weighted-stoilow}
Let $1<p<\infty$ and $\alpha>1$. Suppose
\[
   f=g\circ\phi,
\]
where $g\in H^p$, $\phi:\D\to\D$ is quasiconformal, and $\psi=\phi|_\T$ has boundary weight $w_\psi\in A_p(\T)$. Assume also that $g^*\in L^p(w_\psi|d\zeta|)$ and that, for some $\beta>1$,
\begin{equation}\label{eq:cone}
   \phi(\Gamma_\alpha(\xi))\subset \Gamma_\beta(\psi(\xi))
   \qquad (\xi\in\T).
\end{equation}
Then $f$ has non-tangential boundary values
\[
   f^*(\xi)=g^*(\psi(\xi))
\]
for a.e. $\xi\in\T$, and every convex curve $\Gamma\subset\D$ satisfies
\begin{equation}\label{eq:weighted-stoilow}
   \int_\Gamma |f(z)|^p\ds(z)
   \le B_\alpha C_{p,\beta}([w_\psi]_{A_p})
      \int_\T |f^*(\xi)|^p\,|d\xi|.
\end{equation}
\end{theorem}

\begin{proof}
By \eqref{eq:cone},
\[
   N_\alpha f(\xi)\le N_\beta g(\psi(\xi)).
\]
Therefore, by the definition of $w_\psi$,
\[
   \int_\T N_\alpha f(\xi)^p\,|d\xi|
   \le
   \int_\T N_\beta g(\eta)^p w_\psi(\eta)\,|d\eta|.
\]
For analytic $H^p$ functions, $N_\beta g$ is pointwise controlled, up to a constant depending on $\beta$, by the Hardy--Littlewood maximal function of $|g^*|$. Since $w_\psi\in A_p$ and $1<p<\infty$, the weighted Hardy--Littlewood maximal theorem gives
\[
   \int_\T N_\beta g(\eta)^p w_\psi(\eta)\,|d\eta|
   \le C_{p,\beta}([w_\psi]_{A_p})
      \int_\T |g^*(\eta)|^p w_\psi(\eta)\,|d\eta|.
\]
The last integral equals
\[
   \int_\T |g^*(\psi(\xi))|^p\,|d\xi|
   =\int_\T |f^*(\xi)|^p\,|d\xi|.
\]
The boundary-value assertion follows because $g$ has non-tangential limits a.e. and $\psi_*(|d\xi|)=w_\psi|d\xi|$ is absolutely continuous. Thus $f$ satisfies the admissible maximal estimate of Definition \ref{def:admissible}. The conclusion follows from Theorem \ref{thm:admissible}.
\end{proof}

\begin{remark}
The assumptions in Theorem \ref{thm:weighted-stoilow} identify the boundary mechanisms needed for a norm comparison after Stoilow factorization. The factorization alone does not imply the right-hand side of \eqref{eq:weighted-stoilow}; the boundary weight and cone-distortion conditions make it possible to compare the analytic boundary norm of $g$ with the boundary norm of $f$.
\end{remark}

A simple concrete version is obtained under a bi-Lipschitz boundary hypothesis.

\begin{definition}\label{def:bilip}
Let $0<p<\infty$, $K\ge1$, $L\ge1$, and let $\alpha,\beta>1$. We say that $f$ belongs to the controlled bi-Lipschitz Stoilow class $\mathcal S^p_{K,L}(\alpha,\beta)$ if
\[
   f=g\circ\phi,
\]
where $g\in H^p$, $\phi:\D\to\D$ is $K$-quasiconformal, the boundary extension $\phi|_\T$ is $L$-bi-Lipschitz with respect to arclength on $\T$, and
\begin{equation}\label{eq:bilip-cone}
   \phi(\Gamma_\alpha(\xi))\subset \Gamma_\beta(\phi(\xi))
   \qquad (\xi\in\T).
\end{equation}
\end{definition}

\begin{theorem}[Controlled bi-Lipschitz Stoilow class]\label{thm:bilip}
Let $0<p<\infty$, $K\ge1$, $L\ge1$, and $\alpha,\beta>1$. There exists
\[
   A=A(p,L,\alpha,\beta)<\infty
\]
such that every $f\in\mathcal S^p_{K,L}(\alpha,\beta)$ and every convex curve $\Gamma\subset\D$ satisfy
\begin{equation}\label{eq:bilip}
   \int_\Gamma |f(z)|^p\ds(z)
   \le A\int_\T |f^*(\xi)|^p\,|d\xi|.
\end{equation}
\end{theorem}

\begin{proof}
Write $f=g\circ\phi$ as in Definition \ref{def:bilip}. The cone condition \eqref{eq:bilip-cone} gives
\[
   N_\alpha f(\xi)\le N_\beta g(\phi(\xi)).
\]
Since $\phi|_\T$ is $L$-bi-Lipschitz, the push-forward of arclength by $\phi|_\T$ has density bounded above and below by constants depending only on $L$. Therefore
\[
   \int_\T N_\alpha f(\xi)^p\,|d\xi|
   \le C(L)\int_\T N_\beta g(\eta)^p\,|d\eta|.
\]
The analytic non-tangential maximal theorem gives, for every $0<p<\infty$,
\[
   \int_\T N_\beta g(\eta)^p\,|d\eta|
   \le C(p,\beta)\int_\T |g^*(\eta)|^p\,|d\eta|.
\]
Using the lower measure bound from the bi-Lipschitz hypothesis,
\[
   \int_\T |g^*(\eta)|^p\,|d\eta|
   \le C(L)\int_\T |g^*(\phi(\xi))|^p\,|d\xi|
   =C(L)\int_\T |f^*(\xi)|^p\,|d\xi|.
\]
Thus $f$ satisfies the admissible maximal estimate with a constant depending only on $p,L,\alpha,\beta$. Theorem \ref{thm:admissible} proves \eqref{eq:bilip}.
\end{proof}

\begin{remark}
The cone condition is written as a hypothesis in Definition \ref{def:bilip}. This avoids hiding a boundary-regularity issue inside the phrase ``Stoilow factorization''. In applications one may verify \eqref{eq:bilip-cone} from additional regularity of the quasiconformal factor, but it is not the factorization itself that provides the boundary-norm comparison.
\end{remark}

\begin{corollary}[One-sided Stoilow estimate]\label{cor:one-sided}
Let $0<p<\infty$ and let $f=g\circ\phi$, where $g\in H^p$ and $\phi:\D\to\D$ is quasiconformal. Assume that, for some $\alpha,\beta>1$,
\[
   \phi(\Gamma_\alpha(\xi))\subset \Gamma_\beta(\phi(\xi))
   \qquad (\xi\in\T),
\]
and suppose that $\phi^{-1}|_\T$ is $L$-Lipschitz. Then for every convex curve $\Gamma\subset\D$,
\[
   \int_\Gamma |f(z)|^p\ds(z)
   \le C(p,L,\alpha,\beta)
      \int_\T |g^*(\eta)|^p\,|d\eta|.
\]
\end{corollary}

\begin{proof}
The cone condition gives $N_\alpha f(\xi)\le N_\beta g(\phi(\xi))$. The Lipschitz condition on $\phi^{-1}|_\T$ gives the one-sided change-of-variables estimate
\[
   \int_\T N_\alpha f(\xi)^p\,|d\xi|
   \le C(L)\int_\T N_\beta g(\eta)^p\,|d\eta|.
\]
The analytic non-tangential maximal theorem and Corollary \ref{cor:curve-max} complete the proof.
\end{proof}

\section{Open problems}

The results above leave several natural sharpness and boundary-regularity questions.

\begin{problem}[Sharp harmonic $K$-quasiregular constant]
For harmonic $K$-quasiregular mappings in $h^2$, Theorem \ref{thm:main} gives
\[
   A^{\mathrm{harm}}_{2,K}\le \frac{4K^2}{K^2+1}.
\]
Is this bound sharp for $1<K<\infty$? If it is not sharp, determine the correct dependence on $K$. In either case, identify extremal or asymptotically extremal families.
\end{problem}

\begin{problem}[Quasiregular Hardy classes with boundary control]
Find natural subclasses of the radial quasiregular Hardy class \eqref{eq:radial-class} for which almost everywhere non-tangential boundary values and an estimate of the form
\[
   \|N_\alpha f\|_{L^p(\T)}
   \le C\|f^*\|_{L^p(\T)}
\]
hold. For such classes, determine the best Gabriel constant and its dependence on $p$ and $K$.
\end{problem}

\begin{problem}[Boundary distortion in Stoilow factorization]
Can the $A_p$ boundary-weight hypothesis in Theorem \ref{thm:weighted-stoilow} be weakened? In particular, is an $A_\infty$ condition, a reverse-H\"older condition, or a condition depending only on the quasisymmetry data of $\psi=\phi|_\T$ sufficient for a boundary-norm Gabriel inequality, possibly together with an appropriate non-tangential approach-region condition?
\end{problem}

\begin{problem}[Endpoint exponents]
Das's theorem shows that the full harmonic convex-curve inequality fails in general for $0<p\le1$. Which additional quasiregular or harmonic quasiregular assumptions restore a Gabriel-type inequality at $p=1$ or below?
\end{problem}

\section*{Statements and Declarations}

\noindent\textbf{Funding.}
This research received no specific grant from any funding agency in the public, commercial, or not-for-profit sectors.

\medskip
\noindent\textbf{Competing interests.}
The author declares no competing interests.

\medskip
\noindent\textbf{Data availability.}
No datasets were generated or analysed during the current study.

\end{document}